\documentclass[amstex,12pt]{article}
\usepackage{}
\usepackage{amsmath}
\usepackage{amssymb,amsmath,amsthm}
\newtheorem{theorem}{Theorem}[section]

\newtheorem{lemma}{Lemma}[section]
\newtheorem{definition}{Definition}[section]
\newtheorem{remark}{Remark}[section]

\newcommand{\beq}{\begin{equation}}
\newcommand{\eeq}{\end{equation}}
\newcommand{\beqn}{\begin{eqnarray}}
\newcommand{\eeqn}{\end{eqnarray}}

\def\d{{\delta}}

\allowdisplaybreaks

\begin{document}

\title{Permanence and  almost periodic solutions for a single-species system with impulsive effects on time scales\thanks{This work is supported by the National
Natural Sciences Foundation of People's Republic of China under
Grant 11361072.}}
\author {Yongkun Li\thanks{%
The corresponding author.}, Pan Wang and Bing Li\\
Department of Mathematics, Yunnan University\\
Kunming, Yunnan 650091\\
 People's Republic of China}
\date{}
\maketitle{}
\begin{abstract}
In this paper, we first propose a single-species system with impulsive effects on time scales and by establishing some new comparison theorems of impulsive  dynamic equations  on time scales, we obtain sufficient conditions  to guarantee the permanence of the system.
Then we prove a Massera type theorem for impulsive  dynamic equations  on time scales and based on this theorem,  we establish a criterion for the existence and uniformly asymptotic stability of unique positive almost periodic solution of the system. Finally, we give an example  to show the feasibility of our main results. Our example also shows that the continuous time system and its corresponding discrete time system have the same dynamics. Our results of this paper are completely new.
\end{abstract}

\section{Introduction}

\setcounter{equation}{0}
{\setlength\arraycolsep{2pt}}
 \indent

In 1978, Ludwig et al. \cite{LJH} introduced the following single-species system:
{\setlength\arraycolsep{2pt}\begin{eqnarray*}
x'(t)=x(t)[a-bx(t)]-h(x),
\end{eqnarray*}}
where $x(t)$ is the density of species $x$ at time $t$, $a$ and $b$ are the intrinsic growth rate and self-inhibition rate, respectively. The
$h(x)$-term represents predation. Predation is an increasing function and usually saturates for large enough $x$. If the density of
species $x$ is small, then the predation term $h(x)$ drops rapidly. To investigate the effects of other specific forms of $h(x)$, Murray \cite{M1} took  $h(x)=\frac{\alpha x^{2}(t)}{\beta+x^{2}(t)}$  and the authors of \cite{TLGL, TLC} took $h(x)=\frac{cx(t)}{d+x(t)}$.

Since, in reality, many natural and man-made factors (e.g., fire, drought, flooding deforestation, hunting, harvesting,
breeding etc.) always lead to rapid decrease or increase of population number at fixed times--such sudden changes can often
be characterized mathematically in the form of impulses,   the authors of \cite{TLC} considered the following single-species system governed by the impulsive differential equation:
{\setlength\arraycolsep{2pt}
\begin{eqnarray}\label{e11}
 \left\{%
\begin{array}{lcrcl}
x'(t)=x(t)[a(t)-b(t)x(t)]-\displaystyle\frac{c(t)x(t)}{d(t)+x(t)},\,\,t\neq t_{k},\\
x(t_{k}^{+})=(1+\lambda_{k})x(t_{k}),\,\,t=t_{k},\,\,k\in\mathbb{N},
\end{array}\right.
\end{eqnarray}}
where $x(0)>0$, $t_{k}$ is an impulsive point for every $k$ and $0\leq t_{0} <t_{1}<t_{2}<\ldots<t_{k}<\ldots$, and $\mathbb{N}$ is the set of positive integers, the coefficients $a(t)$, $b(t)$, $c(t)$, $d(t)$ are positive continuous $T$-periodic functions for $t\geq0$, the jump condition reflects the possibility of
impulsive perturbations on species $x$,   $\{\lambda_{k}\}$ is assumed to be a real sequence with $\lambda_{k}>-1$ and there exists an integer $q> 0$ such that
$\lambda_{k+q}=\lambda_{k}$, $t_{k+q}=t_{k}+T$.  By using Brouwer¡¯s fixed point theorem and the Lyapunov function, sufficient conditions for
the existence and global asymptotic stability of positive periodic solutions of the system were derived.

It is well known that biological and environmental parameters are naturally subject to fluctuation in time, the effects of a periodically or almost periodically varying environment are considered as important selective forces on systems in a fluctuating environment. Therefore, on the one hand, models should take into account the seasonality of the periodically changing environment.
However, on the other hand, if the various constituent components of the temporally nonuniform environment is with incommensurable
(nonintegral multiples) periods, then one has to consider the environment to be almost periodic since
there is no a priori reason to expect the existence of periodic solutions. For this reason, the assumption of
almost periodicity is more realistic, more important and more general when we consider the effects of the
environmental factors. Also, at present, few results are available for the existence of positive almost periodic solutions to population models with impulses.

Meanwhile, discrete time models governed by difference equations are very important in implementation and applications, so it is significant to study the discrete time models. As we known, the study of dynamical systems on time scales can unify and extend continuous and discrete analysis \cite{BP1}, which is  now an active area of research. In recent years, a variety of dynamic equations on time scales have been investigated (see \cite{ABLZ,BP1,BP2,BT,HF,LW, PKB,SF, ZL}). However, only few papers \cite{LYZ,ZZ,ZDL}  published on the permanence  for dynamic equation models on time scales, and up to now, there is no paper published on the permanence for impulsive dynamic equation models on time scales.  Thus, it is worthwhile continuing to study the single-species system with impulsive effects on time scales.

Motivated by the above reasons, in this paper, we are concerned with the following single-species system with impulsive effects on time scales:
{\setlength\arraycolsep{2pt}
\begin{eqnarray}\label{e12}
 \left\{%
\begin{array}{lcrcl}
x^{\Delta}(t)=a(t)-b(t)e^{x(t)}-\displaystyle\frac{c(t)}{d(t)+m(t)e^{x(t)}},\,\,t\neq t_{k},\,\,t\in[t_{0}, +\infty)_{\mathbb{T}},\\
x(t_{k}^{+})=x(t_{k})\ln(1+\lambda_{k}),\,\,t=t_{k},\,\,k\in\mathbb{N},
\end{array}\right.
\end{eqnarray}}
where $\mathbb{T}$ is an almost periodic time scale, $0\leq t_{0}\in\mathbb{T}$.

\begin{remark}\label{r1}
Let $y(t)=e^{x(t)}$, if $\mathbb{T}=\mathbb{R}$, then system $\eqref{e12}$ is reduced to the following system:
{\setlength\arraycolsep{2pt}
\begin{eqnarray}\label{eq13}
 \left\{%
\begin{array}{lcrcl}
y'(t)=y(t)[a(t)-b(t)y(t)]-\displaystyle\frac{c(t)y(t)}{d(t)+m(t)y(t)},\,\,t\neq t_{k},\,\,t\in[t_{0}, +\infty),\\
y(t_{k}^{+})=(1+\lambda_{k})y(t_{k}),\,\,t=t_{k},\,\,k\in\mathbb{N},
\end{array}\right.
\end{eqnarray}}
let $m(t)\equiv1$, then system $\eqref{e12}$ is reduced to system $\eqref{e11}$.
If $\mathbb{T}=\mathbb{Z}$, then system $\eqref{e12}$ is reduced to the following system:
{\setlength\arraycolsep{2pt}
\begin{eqnarray}\label{e14}
 \left\{%
\begin{array}{lcrcl}
y(n+1)=y(n)\exp\bigg\{a(n)-b(n)y(n)-\displaystyle\frac{c(n)}{d(n)+m(n)y(n)}\bigg\},\,\,n\neq n_{k},\\
y(n_{k}^{+})=(1+\lambda_{k})y(n_{k}),\,\,n=n_{k},\,\,k\in\mathbb{N},\,\,n\in[t_{0}, +\infty)_{\mathbb{Z}}.
\end{array}\right.
\end{eqnarray}}
\end{remark}

From the point of view of biology, we focus our discussion on the positive solutions of system $\eqref{e12}$. So it is assumed that the initial condition  of system $\eqref{e12}$ is of the form
$x(t_{0})>0.$

For convenience, we denote
\begin{eqnarray*}
 f^l=\inf_{t\in{\mathbb{T}}}f(t),\quad f^u=\sup_{t\in{\mathbb{T}}}f(t),
\end{eqnarray*}
where $f$ is an almost periodic function on $\mathbb{T}$.

Throughout this paper, we assume that
\begin{enumerate}
  \item [$(H_{1})$]
  $a(t), b(t), c(t), d(t), m(t)$ are all bounded nonnegative almost periodic functions on $\mathbb{T}$ such that $a^{l}>0$, $b^{l}>0$, $c^{l}>0$, $d^{l}\geq1$ and $m^{l}>0$.
  \item [$(H_{2})$]
  $\{\lambda_{k}\}$ is an almost periodic sequence and there exists positive constant $r$ such that $0<r\leq\prod_{t_{0}<t_{k}<t}\ln(1+\lambda_{k})\leq 1$ for $t\geq t_{0}$ and $0<\lambda_{k}\leq e-1$ for $k\in\mathbb{N}$;
  \item [$(H_{3})$]
  the set of sequences $\{t_{k}^{j}\}$, $t_{k}^{j}=t_{k+j}-t_{k}$, $k, j\in\mathbb{N}$ is uniformly almost periodic and $\inf_{k}t_{k}^{1}=\theta>0$;
  \item [$(H_{4})$]
  $a^{u}>b^{l}$, $(a^{l}-c^{u})r>b^{u}$, $-a^{l}+c^{u}\in\mathcal{R}^{+}$.
\end{enumerate}

The main purpose of this paper is to discuss the permanence of system $\eqref{e12}$ by establishing some new comparison theorems of impulsive dynamic equations on time scales and based on the permanence result, by   establishing  a  Massera \cite{M} type theorem of impulsive dynamic equations on time scales, we obtain the existence and uniformly asymptotic stability of unique positive almost periodic
solution of system $\eqref{e12}$ on time scales. To the best of our knowledge, this is the first time to study the permanence and almost periodicity of system \eqref{eq13} and system \eqref{e14}, and is the first time to  study the permanence of impulsive dynamic equations on time scales.

The organization of this paper is as follows: In Section 2, we introduce some notations and definitions, state some preliminary results which are needed
in later sections and establish some new comparison theorems. In Section 3, we establish some sufficient conditions for the permanence of $\eqref{e12}$. In Section 4, we prove
 a Massera¡¯s type theorem for impulsive  dynamic equations on time scales and apply this theorem to obtain some sufficient conditions for the existence and uniformly asymptotic stability of unique positive almost periodic solution of $\eqref{e12}$. In Section 5, we give an example to illustrate the feasibility and effectiveness of our results obtained in previous sections. We draw a conclusion in Section 6.

\section{Preliminaries and comparison theorems}
\setcounter{equation}{0}
{\setlength\arraycolsep{2pt}}
 \indent

In this section, we shall recall some basic definitions, lemmas which are used in what follows.

A time scale $\mathbb{T}$ is an arbitrary nonempty closed subset of the real numbers, the forward and backward jump operators $\sigma$, $\rho:\mathbb{T}\rightarrow \mathbb{T}$ and the forward graininess $\mu:\mathbb{T}\rightarrow \mathbb{R}^{+}$ are defined, respectively, by
\[
\sigma(t):=\inf \{s\in\mathbb{T}:s> t\},\,\,\rho(t):=\sup\{s\in\mathbb{T}:s<t\}\,\,
\text{and}\,\,\mu(t)=\sigma(t)-t.
\]

A point $t$ is said to be left-dense if $t>\inf\mathbb{T}$ and $\rho(t)=t$, right-dense if $t<\sup\mathbb{T}$ and $\sigma(t)=t$, left-scattered if $\rho(t)<t$ and right-scattered if $\sigma(t)>t$. If $\mathbb{T}$ has a left-scattered maximum $m$, then $\mathbb{T}^{k}=\mathbb{T}\backslash m$, otherwise $\mathbb{T}^{k}=\mathbb{T}$. If $\mathbb{T}$ has a right-scattered minimum $m$, then $\mathbb{T}_{k}=\mathbb{T}\backslash m$, otherwise $\mathbb{T}^{k}=\mathbb{T}$.

A function $f : \mathbb{T}\rightarrow \mathbb{R}$ is right-dense continuous or rd-continuous provided it is continuous at right-dense points in $\mathbb{T}$ and its left-sided limits exist (finite) at left-dense points in $\mathbb{T}$. If $f$ is continuous at each right-dense point and each left-dense point, then $f$ is said to be a continuous function on $\mathbb{T}$.

For $f:\mathbb{T}\rightarrow\mathbb{R}$ and $t\in{\mathbb{T}^{k}}$, then $f$ is called delta differentiable at $t\in{\mathbb{T}}$ if there exists $c\in\mathbb{R}$ such that for given any $\varepsilon>{0}$, there is an open neighborhood $U$ of  $t$ satisfying
\[
\left|[f(\sigma(t))-f(s)]-c[\sigma(t)-s]\right|\leq\varepsilon\left|\sigma(t)-s\right|,
\]
for all $s\in U$. In this case, $c$ is called the delta derivative of $f$ at $t\in{\mathbb{T}}$, and is denoted by $c=f^{\Delta}(t)$. For $\mathbb{T}=\mathbb{R}$, we have $f^{\Delta}=f^{'}$, the usual derivative, and for $\mathbb{T}=\mathbb{Z}$ we have the backward difference operator, $f^{\Delta}(t)=\Delta f(t):=f(t+1)-f(t)$.

Let $f$ be right-dense continuous, if $F^{\Delta}(t)=f(t)$, then we define the delta integral by
$
\int_{r}^{s}f(t)\Delta t=F(s)-F(r), r, s\in{\mathbb{T}}.
$
\begin{lemma}\cite{BP1}
Assume $f, g : \mathbb{T} \longrightarrow \mathbb{R}$ are delta differentiable at $t\in\mathbb{T}_{¦Ê}$, then
\begin{itemize}
    \item  [$(i)$]  $(f + g)^{\Delta}(t)=f^{\Delta}(t) + g^{\Delta}(t)$;
    \item  [$(ii)$] $(fg)^{\Delta}(t)=f^{\Delta}(t)g(t) + f^{\sigma}(t)g^{\Delta}(t) = f(t)g^{\Delta}(t) + f^{\Delta}(t)g^{\sigma}(t)$;
    \item  [$(iii)$]   If $f$ and $f^{\Delta}$ are continuous, then $(\int_{a}^{t}f(t, s)\Delta s)^{\Delta}=f(\sigma(t), t)+\int_{a}^{t}f^{\Delta}(t, s)\Delta s$.
\end{itemize}
\end{lemma}

A function $p:\mathbb{T}\rightarrow\mathbb{R}$ is called regressive provided $1+\mu(t)p(t)\neq 0$ for all $t\in{\mathbb{T}^{k}}$. The
set of all regressive and rd-continuous functions $p:\mathbb{T}\rightarrow\mathbb{R}$ will be denoted by $\mathcal{R}=\mathcal{R}(\mathbb{T})=\mathcal{R}(\mathbb{T}, \mathbb{R})$. We define the set $\mathcal{R}^{+}=\mathcal{R}^{+}(\mathbb{T}, \mathbb{R})=\{p\in\mathcal{R}: 1+\mu(t)p(t)> 0, \forall t\in{\mathbb{T}}\}$.

If $r\in\mathcal{R}$, then the generalized exponential function $e_{r}$ is defined by
\begin{eqnarray*}
e_{r}(t, s)=\exp\bigg\{\int_s^t\xi_{\mu(\tau)}(r(\tau))\Delta\tau\bigg\},
\end{eqnarray*}
for all $s,t\in\mathbb{T}$, with the cylinder transformation
\begin{eqnarray*}
\xi_h(z)=\bigg\{\begin{array}{ll} {\displaystyle\frac{\mathrm{Log}(1+hz)}{h}},\,\,h\neq 0,\\
z,\,\,\,\,\,\,\,\quad\quad\quad\quad h=0.\\
\end{array}
\end{eqnarray*}

Let $p,q:\mathbb{T}\rightarrow\mathbb{R}$ be two regressive functions, we define
\[
p\oplus q=p+q+\mu pq,\,\,\,\,\ominus p=-\frac{p}{1+\mu p},\,\,\,\,p\ominus q=p\oplus(\ominus q)=\frac{p-q}{1+\mu q}.
\]
Then the generalized exponential function has the following
properties.

\begin{lemma}\cite{BP1}
Assume that $p,q:\mathbb{T}\rightarrow\mathbb{R}$ are two regressive
functions, then
\begin{itemize}
    \item  [$(i)$]  $e_{0}(t,s)\equiv 1$ $\mathrm{and}$ $e_p(t,t)\equiv 1$;
    \item  [$(ii)$] $e_p(\sigma(t),s)=(1+\mu(t)p(t))e_p(t,s)$;
    \item  [$(iii)$]$e_p(t,s)=1/e_p(s,t)=e_{\ominus p}(s,t)$;
    \item  [$(iv)$]  $e_p(t,s)e_p(s,r)=e_p(t,r)$;
    \item  [$(v)$] $e_p(t,s)e_q(t,s)=e_{p\oplus q}(t,s)$;
    \item  [$(vi)$] $e_p(t,s)/e_q(t,s)=e_{p\ominus q}(t,s)$;
    \item  [$(vi)$] $\big(\frac{1}{e_p(t,s)}\big)^{\Delta}=\frac{-p(t)}{e^{\sigma}_p(t,s)}$.
\end{itemize}
\end{lemma}

\begin{lemma}\label{lem23}  Let $f:\mathbb{T}\rightarrow \mathbb{R}$ be a continuously function, $f(t)>0$ and $f^{\Delta}(t)\geq0$ for $t\in\mathbb{T}$, then
\begin{eqnarray*}
\frac{f^{\Delta}(t)}{f^{\sigma}(t)}\leq[\ln (f(t))]^{\Delta}\leq\frac{f^{\Delta}(t)}{f(t)}.
\end{eqnarray*}
If $f(t)>0$ and $f^{\Delta}(t)<0$ for $t\in\mathbb{T}$, then
\begin{eqnarray*}
\frac{f^{\Delta}(t)}{f(t)}\leq[\ln (f(t))]^{\Delta}\leq\frac{f^{\Delta}(t)}{f^{\sigma}(t)}.
\end{eqnarray*}
\end{lemma}
\begin{proof}
If $f^{\Delta}(t)\geq0$ for $t\in\mathbb{T}$, by use of Chain Rule, we can obtain
\begin{eqnarray*}
 \addtolength{\arraycolsep}{-3pt}
[\ln(f(t))]^{\Delta}=
 \left\{%
\begin{array}{lcrcl}
\frac{f^{\Delta}(t)}{f(t)},\,\,\,\,\quad\quad\quad\quad\quad\quad\quad\quad\quad \mu(t)=0,\\
\bigg(\int_{0}^{1}\frac{dh}{f(t)+h\mu(t)f^{\Delta}(t)}\bigg)f^{\Delta}(t),\,\,\,\,\mu(t)\neq0.
\end{array}\right.
\end{eqnarray*}
If $\mu(t)\neq0$, then
\begin{eqnarray}\label{e0.1}
 \addtolength{\arraycolsep}{-3pt}
[\ln(f(t))]^{\Delta}&=&\bigg(\int_{0}^{1}\frac{dh}{f(t)+h\mu(t)f^{\Delta}(t)}\bigg)f^{\Delta}(t)\nonumber\\
&=&\frac{1}{\mu(t)f^{\Delta}(t)}\int_{f(t)}^{f(t)+\mu(t)f^{\Delta}(t)}\frac{ds}{s}f^{\Delta}(t)\nonumber\\
&=&\frac{1}{\mu(t)}\ln\bigg(\frac{f(t)+\mu(t)f^{\Delta}(t)}{f(t)}\bigg)\nonumber\\
&=&\frac{1}{\mu(t)}\ln\bigg(\frac{f^\sigma(t)}{f(t)}\bigg).
\end{eqnarray}

Let
$g_1(r)=r-1-\ln r,$ then $g_1'(r)=\frac{r-1}{r}\geq0$ for $r\geq1$. Hence,
  $g_1$ is an increasing function. By use of $g_1(1)=0$, we have $g_1(r)\geq0$ for $r\geq1$. That is, $r-1\geq\ln r$ for $r\geq1$.
Since
\begin{eqnarray*}
 \addtolength{\arraycolsep}{-3pt}
\frac{f^{\Delta}(t)}{f(t)}&=&\frac{1}{\mu(t)}\bigg(\frac{f(t)+\mu(t)f^{\Delta}(t)}{f(t)}-1\bigg)\\
&\geq&\frac{1}{\mu(t)}\ln\bigg(\frac{f(t)+\mu(t)f^{\Delta}(t)}{f(t)}\bigg)\\
&=&[\ln(f(t))]^{\Delta},
\end{eqnarray*}
then
\begin{eqnarray*}
[\ln (f(t))]^{\Delta}\leq\frac{f^{\Delta}(t)}{f(t)}.
\end{eqnarray*}

Similarly, by use of Chain Rule, if $\mu(t)=0$, then $f(t)=f^\sigma(t)$, we have
\begin{eqnarray*}
[\ln (f(t))]^{\Delta}=\frac{f^{\Delta}(t)}{f(t)}=\frac{f^{\Delta}(t)}{f^\sigma(t)}.
\end{eqnarray*}

If $\mu(t)\neq0$, let $g_2(r)=\ln r+\frac{1}{r}-1,$ then $g_2'(r)=\frac{r-1}{r^2}\geq0$ for $r\geq1$. Hence,
  $g_2$ is an increasing function. By use of $g_2(1)=0$, we have $g_2(r)\geq0$ for $r\geq1$. That is, $\ln r\geq1-\frac{1}{r}$ for $r\geq1$.
By use of $\eqref{e0.1}$, we have
\begin{eqnarray*}
 \addtolength{\arraycolsep}{-3pt}
\frac{f^{\Delta}(t)}{f^\sigma(t)}&=&\frac{1}{\mu(t)}\bigg(\frac{\mu(t)f^{\Delta}(t)-f^\sigma(t)}{f^\sigma(t)}+1\bigg)\\
&=&\frac{1}{\mu(t)}\bigg(1-\frac{f(t)}{f^\sigma(t)}\bigg)\\
&\leq& \frac{1}{\mu(t)}\ln\bigg(\frac{f^\sigma(t)}{f(t)}\bigg)\\
&=&[\ln(f(t))]^{\Delta},
\end{eqnarray*}
so we can obtain
\begin{eqnarray*}
\frac{f^{\Delta}(t)}{f^{\sigma}(t)}\leq[\ln (f(t))]^{\Delta}.
\end{eqnarray*}

Similarly, if $f^{\Delta}(t)\leq0$ for $t\in\mathbb{T}$, we can prove that
\begin{eqnarray*}
\frac{f^{\Delta}(t)}{f(t)}\leq[\ln (f(t))]^{\Delta}\leq\frac{f^{\Delta}(t)}{f^{\sigma}(t)}.
\end{eqnarray*}
The proof is completed.
\end{proof}

\begin{definition}\cite{LW} A time scale $\mathbb{T}$ is called an almost periodic time scale if
\begin{eqnarray*}
\Pi=\big\{\tau\in\mathbb{R}: t\pm\tau\in\mathbb{T}, \forall t\in{\mathbb{T}}\big\}\neq\{0\}.
\end{eqnarray*}
\end{definition}

\begin{definition} \cite{LW} Let $\mathbb{T}$ be an almost periodic time scale. A function $f\in C(\mathbb{T}, \mathbb{R}^{n})$ is called an
almost periodic function if the $\epsilon$-translation set of $f$
\begin{eqnarray*}
E\{\epsilon, f\}=\{t\in\Pi: |f(t+\tau)-f(t)|<\epsilon, \forall t\in{\mathbb{T}}\}
\end{eqnarray*}
is a relatively dense set in $\mathbb{T}$ for all $\epsilon>0$; that is, for any given $\epsilon>0$, there exists a constant
$l(\epsilon)>0$ such that each interval of length $l(\epsilon)$ contains a $\tau(\epsilon)\in E\{\epsilon, f\}$such that
\begin{eqnarray*}
 |f(t+\tau)-f(t)|<\epsilon, \forall t\in{\mathbb{T}}.
\end{eqnarray*}
$\tau$ is called the $\epsilon$-translation number of $f$.
\end{definition}

\begin{lemma} \cite{LW}
If $f, g\in C(\mathbb{T}, \mathbb{R}^{n})$ are almost periodic, then $fg$, $f+g$ are almost periodic.
\end{lemma}

\begin{lemma} \cite{LW}
If $f\in C(\mathbb{T}, \mathbb{R}^{n})$ is almost periodic, $F(\cdot)$ is uniformly continuous on the value field of $f(t)$, then
$F\circ f$ is almost periodic.
\end{lemma}

\begin{lemma} \cite{LW}
If $f\in C(\mathbb{T}, \mathbb{R}^{n})$ is almost periodic, then $F(t)$ is almost periodic if and only if $F(t)$ is bounded on $\mathbb{T}$, where $F(t)=\int_{0}^{t}f(s)\Delta s$.	
\end{lemma}

\begin{definition} \cite{LYZ}
$\varphi\in C(\mathbb{T},\mathbb{R})$ is said to be asymptotically almost periodic, if
\begin{eqnarray}\label{e21}
\varphi(t)=p(t)+q(t),
\end{eqnarray}
where $p(t)$ is an almost periodic function on $\mathbb{T}$ and $q(t)$ is continuous on $\mathbb{T}$,
$\lim_{t\rightarrow\infty}q(t)=0$.
\end{definition}

\begin{lemma} \cite{LYZ}
Let $\varphi\in C(\mathbb{T},\mathbb{R})$ is asymptotically almost periodic. Then the decomposition $\eqref{e21}$ is unique.
\end{lemma}

\begin{lemma} \cite{LYZ}
Let $\varphi\in C(\mathbb{T},\mathbb{R})$, the following propositions are equivalent:
\begin{itemize}
    \item  [$(i)$] $\varphi(t)$ is asymptotically almost periodic;
    \item  [$(ii)$] for any $\varepsilon>0$, there exist constants $l(\varepsilon)>0$ and $k(\varepsilon)>0$ such that each interval of length $l(\varepsilon)$ contains at least one $\tau$ such that
        \begin{eqnarray*}
       |\varphi(t+\tau)-\varphi(t)|<\varepsilon,\,\,\forall t,t+\tau\geq k(\varepsilon).
        \end{eqnarray*}
\end{itemize}
\end{lemma}

\begin{lemma} \cite{BP1}
Assume that $a\in\mathcal{R}$ and $t_{0}\in{\mathbb{T}}$, if $a\in\mathcal{R}^{+}$ on $\mathbb{T}^{k}$, then $e_a(t,t_{0})> 0$ for all $t\in{\mathbb{T}}$.
\end{lemma}

\begin{lemma}\cite{HF}\label{lem210}
Assume that $x\in PC^{1}[\mathbb{T}, \mathbb{R}]$ and
\begin{eqnarray*}
 \left\{%
\begin{array}{lcrcl}
x^{\Delta}(t)\leq(\geq) p(t)x(t)+q(t),\,\,t\neq t_{k},\,\,t\in[t_{0}, +\infty)_{\mathbb{T}},\\
x(t_{k}^{+})\leq(\geq) d_{k}x(t_{k})+b_{k},\,\,t=t_{k},\,\,k\in\mathbb{N},
\end{array}\right.
\end{eqnarray*}
then for $t\geq t_{0}\geq0$,
\begin{eqnarray*}
x(t)&\leq(\geq)& x(t_{0})\prod_{t_{0}<t_{k}<t}d_{k}e_{p}(t, t_{0})+\sum_{t_{0}<t_{k}<t}\bigg(\prod_{t_{0}<t_{j}<t}d_{j}e_{p}(t, t_{k})\bigg)b_{k}\\
&&+\int_{t_{0}}^{t}\prod_{s<t_{k}<t}d_{k}e_{p}(t, \sigma(s))q(s)\Delta s,
\end{eqnarray*}
where $PC^{1}=\{y:[0, \infty)_{\mathbb{T}}\rightarrow\mathbb{R}$ which is rd-continuous except at $t_{k}$, $k=1, 2,\ldots$, for which
$y(t_{k}^{-})$, $y(t_{k}^{+})$, $y^{\Delta}(t_{k}^{-})$, $y^{\Delta}(t_{k}^{+})$ exist with $y(t_{k}^{-})=y(t_{k})$, $y^{\Delta}(t_{k}^{-})=y^{\Delta}(t_{k})\}.$
\end{lemma}

\begin{lemma}\label{lem211}
Assume that $x\in PC^{1}[\mathbb{T}, \mathbb{R}]$, $-a\in\mathcal{R}^{+}$, $\alpha\leq\prod_{t_{0}<t_{k}<t}d_{k}\leq \beta$ for $t\geq t_{0}$.
\begin{itemize}
    \item  [$(i)$] If \begin{eqnarray}\label{e23}
 \left\{%
\begin{array}{lcrcl}
x^{\Delta}(t)\leq b-ax(t),\,\,t\neq t_{k},\,\,t\in[t_{0}, +\infty)_{\mathbb{T}},\\
x(t_{k}^{+})\leq d_{k}x(t_{k})+b_{k},\,\,t=t_{k},\,\,k\in\mathbb{N},
\end{array}\right.
\end{eqnarray}
then for $t\geq t_{0}$,
    \begin{eqnarray*}
x(t)\leq x(t_{0})\beta e_{(-a)}(t, t_{0})+\sum_{t_{0}<t_{k}<t}\beta e_{(-a)}(t, t_{k})b_{k}+\frac{b\beta}{a}[1-e_{(-a)}(t, t_{0})].
\end{eqnarray*}
Especially, if $b> 0$, $a> 0$, we have $\limsup_{t\rightarrow+\infty}x(t)\leq\frac{b\beta}{a}$.
    \item  [$(ii)$] If \begin{eqnarray}\label{e24}
 \left\{%
\begin{array}{lcrcl}
x^{\Delta}(t)\geq b-ax(t),\,\,t\neq t_{k},\,\,t\in[t_{0}, +\infty)_{\mathbb{T}},\\
x(t_{k}^{+})\geq d_{k}x(t_{k})+b_{k},\,\,t=t_{k},\,\,k\in\mathbb{N},
\end{array}\right.
\end{eqnarray} then for $t\geq t_{0}$,
    \begin{eqnarray*}
 x(t)\geq x(t_{0})\alpha e_{(-a)}(t, t_{0})+\sum_{t_{0}<t_{k}<t}\alpha e_{(-a)}(t, t_{k})b_{k}+\frac{b\alpha}{a}[1-e_{(-a)}(t, t_{0})].
   \end{eqnarray*}
 Especially, if $b> 0$, $a> 0$, we have $\liminf_{t\rightarrow+\infty}x(t)\geq\frac{b\alpha}{a}$.
\end{itemize}
\end{lemma}
\begin{proof}
Because the proof of $(ii)$ is similar to that of $(i)$, we only prove $(i)$.
By   Lemma \ref{lem210} and $\eqref{e23}$, we have
\begin{eqnarray*}
x(t)&\leq& x(t_{0})\prod_{t_{0}<t_{k}<t}d_{k}e_{(-a)}(t, t_{0})+\sum_{t_{0}<t_{k}<t}\bigg(\prod_{t_{0}<t_{j}<t}d_{j}e_{(-a)}(t, t_{k})\bigg)b_{k}\\
&&+\int_{t_{0}}^{t}\prod_{s<t_{k}<t}d_{k}e_{(-a)}(t, \sigma(s))b\Delta s.
\end{eqnarray*}
In view of $\prod_{t_{0}<t_{k}<t}d_{k}\leq \beta$, we have
\begin{eqnarray*}
x(t)&\leq& x(t_{0})\beta e_{(-a)}(t, t_{0})+\sum_{t_{0}<t_{k}<t}\beta e_{(-a)}(t, t_{k})b_{k}+\beta\int_{t_{0}}^{t}e_{(-a)}(t, \sigma(s))b\Delta s\\
&\leq& x(t_{0})\beta e_{(-a)}(t, t_{0})+\sum_{t_{0}<t_{k}<t}\beta e_{(-a)}(t, t_{k})b_{k}+\frac{b\beta}{a}[1-e_{(-a)}(t, t_{0})].
\end{eqnarray*}
In particular, if $a> 0$, then $e_{(-a)}(t, t_{0})<1$. We obtain $\limsup_{t\rightarrow+\infty}x(t)\leq\frac{b\beta}{a}$. The proof is completed.
\end{proof}

Similarly, we can easily obtain the following results:
\begin{lemma}\label{lem212}
Assume that $-b\in\mathcal{R}^{+}$, $x\in PC^{1}[\mathbb{T}, \mathbb{R}]$ and $x(t)>0$ for $t\in\mathbb{T}$ and $\alpha\leq\prod_{t_{0}<t_{k}<t}d_{k}\leq \beta$ for $t\geq t_{0}$.
\begin{itemize}
    \item  [$(i)$] If \begin{eqnarray}\label{e25}
 \left\{%
\begin{array}{lcrcl}
x^{\Delta}(t)\leq x^{\sigma}(t)(b-ax(t)),\,\,t\neq t_{k},\,\,t\in[t_{0}, +\infty)_{\mathbb{T}},\\
x(t_{k}^{+})\leq d_{k}x(t_{k}),\,\,t=t_{k},\,\,k\in\mathbb{N},
\end{array}\right.
\end{eqnarray} then for $t\geq t_{0}$,
    \begin{eqnarray*}
x(t)\leq \frac{b\beta}{a}\bigg[1+\bigg(\frac{b}{ax(t_{0})}-1\bigg)e_{(-b)}(t, t_{0})\bigg]^{-1}.
   \end{eqnarray*}
Especially, if $b> 0$, $a> 0$, we have $\limsup_{t\rightarrow+\infty}x(t)\leq\frac{b\beta}{a}$.
    \item  [$(ii)$] If \begin{eqnarray}\label{e26}
 \left\{%
\begin{array}{lcrcl}
x^{\Delta}(t)\geq x^{\sigma}(t)(b-ax(t)),\,\,t\in[t_{0}, +\infty)_{\mathbb{T}},\,\,t\neq t_{k},\\
x(t_{k}^{+})\geq d_{k}x(t_{k}),\,\,t=t_{k},\,\,k\in\mathbb{N},
\end{array}\right.
\end{eqnarray} then for $t\geq t_{0}$,
    \begin{eqnarray*}
 x(t)\geq \frac{b\alpha}{a}\bigg[1+\bigg(\frac{b}{ax(t_{0})}-1\bigg)e_{(-b)}(t, t_{0})\bigg]^{-1}.
   \end{eqnarray*}
 Especially, if $b> 0$, $a> 0$, we have $\liminf_{t\rightarrow+\infty}x(t)\geq\frac{b\alpha}{a}$.
\end{itemize}
\end{lemma}
\begin{proof}
Because the proof of $(ii)$ is similar to that of $(i)$, we only prove $(i)$. Noticing that
\begin{eqnarray*}
\frac{x^{\Delta}(t)}{x(t)x^{\sigma}(t)}\leq\frac{b}{x(t)}-a
\end{eqnarray*}
and
\begin{eqnarray*}
\bigg(\frac{1}{x(t)}\bigg)^{\Delta}=-\frac{x^{\Delta}(t)}{x(t)x^{\sigma}(t)}.
\end{eqnarray*}
Let $y(t)=\frac{1}{x(t)}$, we have
\begin{eqnarray*}
 \left\{%
\begin{array}{lcrcl}
-y^{\Delta}(t)\leq by(t)-a,\,\,t\neq t_{k},\,\,t\in[t_{0}, +\infty)_{\mathbb{T}},\\
y(t_{k}^{+})\geq \frac{1}{d_{k}}y(t_{k}),\,\,t=t_{k},\,\,k\in\mathbb{N},
\end{array}\right.
\end{eqnarray*}
that is,
\begin{eqnarray*}
 \left\{%
\begin{array}{lcrcl}
y^{\Delta}(t)\geq a-by(t),\,\,t\neq t_{k},\,\,t\in[t_{0}, +\infty)_{\mathbb{T}},\\
y(t_{k}^{+})\geq \frac{1}{d_{k}}y(t_{k}),\,\,t=t_{k},\,\,k\in\mathbb{N}.
\end{array}\right.
\end{eqnarray*}
By Lemma \ref{lem211} $(ii)$, for $t>t_{0}$, we have
 \begin{eqnarray*}
 y(t)\geq y(t_{0})\frac{1}{\beta}e_{(-b)}(t, t_{0})+\frac{a}{b\beta}[1-e_{(-b)}(t, t_{0})],
   \end{eqnarray*}
that is,
\begin{eqnarray*}
 x(t)\leq \frac{b\beta}{a}\bigg[1+\bigg(\frac{b}{ax(t_{0})}-1\bigg)e_{(-b)}(t, t_{0})\bigg]^{-1}.
   \end{eqnarray*}
 In particular, if $b> 0$, then $e_{(-b)}(t, t_{0})<1$. We obtain $\limsup_{t\rightarrow+\infty}x(t)\leq\frac{b\beta}{a}$. The proof is completed.
\end{proof}

\begin{lemma}\label{lem213}
Assume that $-b\in\mathcal{R}^{+}$, $a>0$, $x\in PC^{1}[\mathbb{T}, \mathbb{R}]$ and $x(t)>0$, $\alpha\leq\prod_{t_{0}<t_{k}<t}d_{k}\leq \beta$ for $t\geq t_{0}$, $\bar{\mu}=\sup_{t\in\mathbb{T}}\mu(t)$.
 If \begin{eqnarray}\label{e27}
 \left\{%
\begin{array}{lcrcl}
x^{\Delta}(t)\geq x(t)(b-ax(t)),\,\,t\neq t_{k},\,\,t\in[t_{0}, +\infty)_{\mathbb{T}},\\
x(t_{k}^{+})\geq d_{k}x(t_{k}),\,\,t=t_{k},\,\,k\in\mathbb{N},
\end{array}\right.
\end{eqnarray} then for $t\geq t_{0}$,
    \begin{eqnarray*}
  x(t)\geq \frac{b\alpha}{a}\bigg[1+\bigg(\frac{b}{ax(t_{0})}-1\bigg)e_{(-\frac{b}{1+\bar{\mu} b})}(t, t_{0})\bigg]^{-1}.
   \end{eqnarray*}
 Especially, if $b> 0$, we have $\liminf_{t\rightarrow+\infty}x(t)\geq\frac{b\alpha}{a}$.
\end{lemma}
\begin{proof}
Since
\begin{eqnarray*}
x(t)=x^{\sigma}(t)-\mu(t)x^{\Delta}(t),
\end{eqnarray*}
we have
\begin{eqnarray*}
x^{\Delta}(t)\geq x(t)(b-ax(t))=(x^{\sigma}(t)-\mu(t)x^{\Delta}(t))(b-ax(t)),
\end{eqnarray*}
this is,
\begin{eqnarray*}
[1+\mu(t)(b-ax(t))]x^{\Delta}(t)\geq x^{\sigma}(t)(b-ax(t)),
\end{eqnarray*}
then
\begin{eqnarray*}
x^{\Delta}(t)\geq x^{\sigma}(t)\bigg(\frac{b}{1+\bar{\mu} b}-\frac{a}{1+\bar{\mu} b}x(t)\bigg),
\end{eqnarray*}
so we have
\begin{eqnarray}\label{e28}
 \left\{%
\begin{array}{lcrcl}
x^{\Delta}(t)\geq x^{\sigma}(t)\bigg(\frac{b}{1+\bar{\mu} b}-\frac{a}{1+\bar{\mu} b}x(t)\bigg),\,\,t\neq t_{k},\,\,t\in[t_{0}, +\infty)_{\mathbb{T}},\\
x(t_{k}^{+})\geq d_{k}x(t_{k}),\,\,t=t_{k},\,\,k\in\mathbb{N}.
\end{array}\right.
\end{eqnarray}
From Lemma \ref{lem212} and $\eqref{e28}$, we have
\begin{eqnarray*}
 x(t)\geq \frac{b\alpha}{a}\bigg[1+\bigg(\frac{b}{ax(t_{0})}-1\bigg)e_{(-\frac{b}{1+\bar{\mu} b})}(t, t_{0})\bigg]^{-1}.
   \end{eqnarray*}

 In particular, if $b> 0$, then $e_{(-\frac{b}{1+\bar{\mu} b})}(t, t_{0})<1$. We obtain $\liminf_{t\rightarrow+\infty}x(t)\geq\frac{b\alpha}{a}$. The proof is completed.
\end{proof}

\section{Permanence}
\setcounter{equation}{0}
{\setlength\arraycolsep{2pt}}
 \indent

In this section, we will give our main results about the permanence of system $\eqref{e12}$. For convenience, we introduce the following notations:
\begin{eqnarray*}
x^{\ast}=\frac{ a^{u}-b^{l}}{b^{l}},\,\,\,\,\,\,
x_{\ast}=\ln\bigg(\frac{(a^{l}-c^{u})r}{b^{u}}\bigg).
\end{eqnarray*}

\begin{lemma}\label{lem31}
Assume that $(H_{1})$-$(H_{4})$ hold. Let $x(t)$ be any solution of system $\eqref{e12}$, then
 \begin{eqnarray*}
x_{\ast}\leq\liminf_{t\rightarrow+\infty}x(t)\leq\limsup_{t\rightarrow+\infty}x(t)\leq x^{\ast}.
 \end{eqnarray*}
\end{lemma}
\begin{proof}
Let $x(t)$ be any solution of system $\eqref{e12}$. From $\eqref{e12}$, it follows that
\begin{eqnarray*}
x^{\Delta}(t)&\leq& a^{u}-b^{l}e^{x(t)}\\
&\leq&a^{u}-b^{l}[1+x(t)]\\
&=&(a^{u}-b^{l})-b^{l}x(t),
\end{eqnarray*}
then
\begin{eqnarray*}
 \left\{%
\begin{array}{lcrcl}
x^{\Delta}(t)\leq(a^{u}-b^{l})-b^{l}x(t),\,\,t\neq t_{k},\,\,t\in[t_{0}, +\infty)_{\mathbb{T}},\\
x(t_{k}^{+})\leq x(t_{k})\ln(1+\lambda_{k}),\,\,t=t_{k},\,\,k\in\mathbb{N}.
\end{array}\right.
\end{eqnarray*}
In view of Lemma \ref{lem211} $(i)$, we have
\begin{eqnarray*}
\limsup_{t\rightarrow+\infty}x(t)\leq \frac{ a^{u}-b^{l}}{b^{l}}=x^{\ast}.
\end{eqnarray*}
By  system $\eqref{e12}$, we arrive at
\begin{eqnarray*}
x^{\Delta}(t)\geq a^{l}-c^{u}-b^{u}e^{x(t)},\,\,t\in[t_{0}, +\infty)_{\mathbb{T}}.
\end{eqnarray*}
 Let $N(t)=e^{x(t)}$, obviously $N(t)>0$, above inequality yields that
\begin{eqnarray*}
[\ln(N(t))]^{\Delta}\geq a^{l}-c^{u}-b^{u}N(t).
\end{eqnarray*}
If $N^{\Delta}(t)\geq0$, in view of Lemma \ref{lem23}, we have
\begin{eqnarray*}
\frac{N^{\Delta}(t)}{N(t)}\geq a^{l}-c^{u}-b^{u}N(t),
\end{eqnarray*}
 then
\begin{eqnarray*}
N^{\Delta}(t)\geq N(t)[ a^{l}-c^{u}-b^{u}N(t)].
\end{eqnarray*}
Thus
\begin{eqnarray*}
 \left\{%
\begin{array}{lcrcl}
N^{\Delta}(t)\geq N(t)[ a^{l}-c^{u}-b^{u}N(t)],\,\,t\neq t_{k},\,\,t\in[t_{0}, +\infty)_{\mathbb{T}},\\
N(t_{k}^{+})\geq (1+\lambda_{k})N(t_{k}),\,\,t=t_{k},\,\,k\in\mathbb{N}.
\end{array}\right.
\end{eqnarray*}
By applying Lemma \ref{lem213}, $(a^{l}-c^{u})r>b^{u}$ and $-a^{l}+c^{u}\in\mathcal{R}^{+}$, we have
\begin{eqnarray*}
\liminf_{t\rightarrow+\infty}N(t)\geq \frac{(a^{l}-c^{u})r}{b^{u}}.
\end{eqnarray*}

If $N^{\Delta}(t)<0$, in view of Lemma \ref{lem23}, we have
\begin{eqnarray*}
\frac{N^{\Delta}(t)}{N^\sigma(t)}\geq a^{l}-c^{u}-b^{u}N(t),
\end{eqnarray*}
 then
\begin{eqnarray*}
N^{\Delta}(t)\geq N^\sigma(t)[ a^{l}-c^{u}-b^{u}N(t)].
\end{eqnarray*}
Thus
\begin{eqnarray*}
 \left\{%
\begin{array}{lcrcl}
N^{\Delta}(t)\geq N^\sigma(t)[ a^{l}-c^{u}-b^{u}N(t)],\,\,t\neq t_{k},\,\,t\in[t_{0}, +\infty)_{\mathbb{T}},\\
N(t_{k}^{+})\geq (1+\lambda_{k})N(t_{k}),\,\,t=t_{k},\,\,k\in\mathbb{N}.
\end{array}\right.
\end{eqnarray*}
By applying Lemma \ref{lem213}, $(a^{l}-c^{u})r>b^{u}$ and $-a^{l}+c^{u}\in\mathcal{R}^{+}$, we have
\begin{eqnarray*}
\liminf_{t\rightarrow+\infty}N(t)\geq \frac{(a^{l}-c^{u})r}{b^{u}}.
\end{eqnarray*}
 That is,
\begin{eqnarray*}
\liminf_{t\rightarrow+\infty}x(t)\geq \ln\bigg(\frac{(a^{l}-c^{u})r}{b^{u}}\bigg)=x_{\ast}.
\end{eqnarray*}
The proof is complete.
\end{proof}
\begin{theorem}\label{thm31}
Assume that $(H_{1})$-$(H_{4})$ hold, then system $\eqref{e12}$ is permanence.
\end{theorem}

\section{A Massera type theorem and   almost periodic solutions}

\setcounter{equation}{0}
{\setlength\arraycolsep{2pt}}
 \indent

In this section, we will prove a Massera type theorem and use it to
study the existence of almost periodic solutions of $\eqref{e12}$. Consider the following equation
\begin{eqnarray}\label{e41}
 \addtolength{\arraycolsep}{-3pt}
 \left\{%
\begin{array}{lcrcl}
x^{\Delta}(t)=f(t,x),\,\,\,\, t\neq t_{k},\,\,t\in\mathbb{T}^{+},\\
\Delta x(t_{k})=I_{k}(x(t_{k})),\,\,\,\, t= t_{k},\,\,k\in\mathbb{N},
\end{array}\right.
\end{eqnarray}
where $f:\mathbb{T}^{+} \times S_{B}\rightarrow \mathbb{R}$, $S_{B}= \{x\in \mathbb{R} : \|x\|< B\}, \|x\|=\sup_{t\in\mathbb{T}}|x(t)|$, the functions $I_{k}\in C[\mathbb{R}, \mathbb{R}]$, $k\in\mathbb{N}$ are almost periodic uniformly with respects to $x\in S_{B}$ and are Lipschitz continuous in $x$, $f(t, x)$ is almost periodic in $t$ uniformly for $x\in S_{B}$ and is continuous in $x$. The set of sequences $\{t_{k}^{j}\}$, $t_{k}^{j}=t_{k+j}-t_{k}$, $k, j\in\mathbb{N}$ is uniformly almost periodic and $\inf_{k}t_{k}^{1}=\theta>0$. To find the solution of $\eqref{e41}$, we consider the product system of $\eqref{e41}$ as follows
\begin{eqnarray*}
 \left\{%
\begin{array}{lcrcl}
x^{\Delta}(t)=f(t,x),\,\,\,\, t\neq t_{k},\,\,t\in\mathbb{T}^{+},\\
\Delta x(t_{k})=I_{k}(x(t_{k})),\,\,\,\, t= t_{k},\,\,k\in\mathbb{N},
\end{array}\right.\quad\quad
 \left\{%
\begin{array}{lcrcl}
y^{\Delta}(t)=f(t,y),\,\,\,\, t\neq t_{k},\,\,t\in\mathbb{T}^{+},\\
\Delta y(t_{k})=I_{k}(y(t_{k})),\,\,\,\, t= t_{k},\,\,k\in\mathbb{N}.
\end{array}\right.
\end{eqnarray*}

Define
$V_{1}=\Big\{V:\mathbb{T}^{+}\times S_{B}\times S_{B}\rightarrow\mathbb{R}^{+}, V$ is rd-continuous in $(t_{k-1}, t_{k}]_{\mathbb{T}^{+}}\times S_{B}\times S_{B}$ and
 $\lim\limits_{(t, x, y)\rightarrow(t_{k}, x_{0}, y_{0}),t>t_{k}}V(t, x, y)=V(t_{k}^{+}, x_{0}, y_{0})\Big\}.$

\begin{lemma}\label{lem41}
Suppose that there exists a Lyapunov functional $V(t, x, y)\in V_{1}$ satisfying the following conditions
\begin{itemize}
    \item  [$(i)$]  $a(||x-y||)\leq V(t, x, y)\leq b(||x-y||)$, where $(t, x, y)\in\mathbb{T}^{+}\times S_{B}\times S_{B}$, $a, b\in \kappa$, $\kappa=\{a\in C(\mathbb{R}^{+}, \mathbb{R}^{+}): a(0)=0$ and $a$ is increasing\};
    \item  [$(ii)$] $|V(t, x, y)-V(t, x_{1}, y_{1})|\leq L(\|x-x_{1}\|+ \|y-y_{1}\|)$, where $(t, x, y)\in\mathbb{T}^{+}\times S_{B}\times S_{B}$, $L > 0$ is a constant;
    \item  [$(iii)$] $V(t_{k}^{+}, x+I_{k}(x), y+I_{k}(y))\leq V(t, x, y)$, $x, y\in S_{B}$, $t=t_{k}$, $k\in\mathbb{N};$
    \item  [$(iv)$] $D^{+}V_{(4.1)}^{\Delta}(t, x, y)\leq -cV(t, x, y)$, where $c> 0$, $-c\in\mathcal{R}^{+}$, $x, y\in S_{B}$, $t\neq t_{k}$, $k\in\mathbb{N}$.
\end{itemize}
Moreover, if there exists a solution $x(t)\in S$ of $\eqref{e41}$ for $t\in\mathbb{T}^{+}$, where $S\subset S_{B}$ is a compact set, then there exists a unique almost periodic solution $p(t)\in S$ of $\eqref{e41}$, which is uniformly asymptotically stable. In particular, if $f (t, x)$ is $\omega$-periodic in $t$ uniformly for $x\in S_{B}$ and there exists a positive integer $q$ such that $t_{k+q}=t_{k}+\omega$, $I_{k+q}(x)=I_{k}(x)$ with $t_{k}\in\mathbb{T}^{+}$, then $p(t)$ is also periodic.
\end{lemma}
\begin{proof}
Take $\{\omega_{n}\}\subset\Pi$ such that $\omega_{n}\rightarrow+\infty$ as $n\rightarrow+\infty$. Suppose that $\varphi(t)\in S$ is a solution of $\eqref{e41}$ for $t\in\mathbb{T}^{+}$, then $\varphi(t+\omega_{n})\in S$ is a solution of the following equation
\begin{eqnarray*}
 \left\{%
\begin{array}{lcrcl}
x^{\Delta}(t)=f(t+\omega_{n},x),\,\,\,\, t\neq t_{k}-\omega_{n},\\
\Delta x(t_{k}+\omega_{n})=I_{k}(x(t_{k}+\omega_{n})),\,\,\,\, t= t_{k}-\omega_{n}.
\end{array}\right.
\end{eqnarray*}
For any $\varepsilon>0$, take large enough integer $n_{0}(\varepsilon, \beta)$ such that when $m\geq l\geq n_{0}(\varepsilon)$, we have
\begin{eqnarray*}
b(2B)e_{(-c)}(\omega_{l}, 0)<\frac{a(\varepsilon)}{2}
\end{eqnarray*}
and
\begin{eqnarray*}
|f(t+\omega_{m},x)-f(t+\omega_{l},x)|<\frac{ca(\varepsilon)}{2L}.
\end{eqnarray*}
Then for $(iv)$, we have
\begin{eqnarray*}
&&D^{+}V^{\Delta}(t, \varphi(t), \varphi(t+\omega_{m}-\omega_{l}))\\
&\leq&-cV(t, \varphi(t), \varphi(t+\omega_{m}-\omega_{l}))+L|f(t+\omega_{m}-\omega_{l},\varphi(t+\omega_{m}-\omega_{l}))\\
&&-f(t,\varphi(t+\omega_{m}-\omega_{l}))|,
\end{eqnarray*}
for $t\neq t_{k}-(\omega_{m}-\omega_{l})$.

On the other hand, from $t= t_{k}-(\omega_{m}-\omega_{l})$ and $(iii)$ it follows that
\begin{eqnarray*}
&&V(t, \varphi(t)+I_{k}(\varphi(t)), \varphi(t+\omega_{m}-\omega_{l})+I_{k}(\varphi(t+\omega_{m}-\omega_{l})))\leq V(t, \varphi(t), \varphi(t+\omega_{m}-\omega_{l})),
\end{eqnarray*}
then
\begin{eqnarray*}
&&D^{+}V^{\Delta}(t, \varphi(t), \varphi(t+\omega_{m}-\omega_{l}))\\
&\leq&-cV(t, \varphi(t), \varphi(t+\omega_{m}-\omega_{l}))+\frac{ca(\varepsilon)}{2}
\end{eqnarray*}
when $m\geq l\geq n_{0}(\varepsilon)$, we have
\begin{eqnarray*}
&&V(t+\omega_{l}, \varphi(t+\omega_{l}), \varphi(t+\omega_{m}))\\
&\leq&e_{(-c)}(t+\omega_{l}, 0)V(0, \varphi(0), \varphi(\omega_{m}-\omega_{l}))+\frac{a(\varepsilon)}{2}(1-e_{(-c)}(t+\omega_{l}, 0))\\
&\leq&e_{(-c)}(t+\omega_{l}, 0)V(0, \varphi(0), \varphi(\omega_{m}-\omega_{l}))+\frac{a(\varepsilon)}{2}\\
&\leq&a(\varepsilon).
\end{eqnarray*}
By $(i)$, for $m\geq l\geq n_{0}(\varepsilon)$ and $t\in\mathbb{T}^{+}$, we obtain
\begin{eqnarray*}
|\varphi(t+\omega_{m})-\varphi(t+\omega_{l})|<\varepsilon,
\end{eqnarray*}
which implies that $\varphi(t)$ is asymptotically almost periodic. Then $\varphi(t)=p(t)+q(t)$,
where $p(t)$ is almost periodic and $q(t)\rightarrow0$, as $t\rightarrow\infty$. Therefore $p(t)\in S$ is an almost
periodic solution of $\eqref{e41}$. It is easy to verify that $p(t)$ is uniformly asymptotically
stable and every solution in $S_{B}$ tends to $p(t)$, which means that $p(t)$ is unique.
In particular, if $f(t, x)$ is $\omega$-periodic in $t$ uniformly for $x\in S_{B}$ and there exists a
positive integer $q$ such that $t_{k+q}=t_{k}+\omega$, $I_{k+q}(x)=I_{k}(x)$ with $t_{k}\in\mathbb{T}^{+}$, then $p(t+\omega)\in S$ is
also a solution. By the uniqueness, we have $p(t+\omega)=p(t)$. The proof is complete.
\end{proof}

Let $x(t)$ be any solution of system $\eqref{e12}$, $\Omega= \{x(t): 0 <x_{\ast}\leq x(t)\leq  x^{\ast}\}$. It is easy to verify that under the conditions of Theorem \ref{thm31}, $\Omega$ is an invariant set of $\eqref{e12}$.
\begin{lemma}\label{lem42}
Assume that $(H_{1})$-$(H_{4})$ hold, then $\Omega\neq\emptyset$.
\end{lemma}
\begin{proof}
By the almost periodicity of $a(t)$, $b(t)$, $c(t)$, $d(t)$ and $m(t)$, there exists   a  sequence $\omega=\{\omega_{p}\}\subseteq \Pi$ with $\omega_{p}\rightarrow+\infty$ as $p\rightarrow+\infty$ such that for $t\neq t_{k}$, we have
\begin{eqnarray*}
a(t+\omega_{p})\rightarrow a(t),\,\,b(t+\omega_{p})\rightarrow b(t),\,\,c(t+\omega_{p})\rightarrow c(t)\\d(t+\omega_{p})\rightarrow d(t),\,\,m(t+\omega_{p})\rightarrow m(t),\,\,p\rightarrow+\infty,
\end{eqnarray*}
and there exists a subsequence $\{k_l\}$ of $\{p\}$, $k_l\rightarrow+\infty$, $l\rightarrow+\infty$, such that $t_{k_l}\rightarrow t_{k}$, $\lambda_{k_l}\rightarrow \lambda_{k}$.

In view of Lemma \ref{lem31}, for all $\varepsilon>0$ then there exists a $t_{1}\in\mathbb{T}$ and $t_{1}\geq t_{0}$ such that
\begin{eqnarray*}
 x_{\ast}-\epsilon\leq x(t)\leq  x^{\ast}+\epsilon,\,\,{\rm for}\ t\geq t_{1}.
\end{eqnarray*}
Write $x_{p}(t)=x(t+\omega_{p})$ for $t\geq t_{1}$, $p= 1, 2,\ldots$.
For any positive integer $q$, it is easy to see that there exist sequences $\{x_{p}(t):p\geq q\}$ such that the sequences $\{x_{p} (t)\}$ has subsequences, denoted by $\{x_{p} (t)\}$ again, converging on any finite interval of $\mathbb{T}$ as $p\rightarrow+\infty$, respectively. Thus, there is a   function $y(t)$ defined on $\mathbb{T}$ such that
\begin{eqnarray*}
x_{p}(t)\rightarrow y(t),\,\,{\rm for}\ t\in\mathbb{T},\,\,{\rm as}\ p\rightarrow+\infty.
\end{eqnarray*}
Since
{\setlength\arraycolsep{2pt}
\begin{eqnarray*}
 \left\{%
\begin{array}{lcrcl}
x_{p}^{\Delta}(t)=a(t+\omega_{p})-b(t+\omega_{p})e^{x(t+\omega_{p})}\\
\quad\quad\quad\,\,\,\,-\displaystyle\frac{c(t+\omega_{p})}{d(t+\omega_{p})+m(t+\omega_{p})e^{x(t+\omega_{p})}},\,\,t\neq t_{k}-\omega_{p},\,\,t\in[t_{0}, +\infty)_{\mathbb{T}},\\
x_{p}(t_{k}^{+})=x(t_{k}+\omega_{p})\ln(1+\lambda_{k_l}),\,\,t=t_{k}-\omega_{p},\,\,k\in\mathbb{N},
\end{array}\right.
\end{eqnarray*}}
we have
{\setlength\arraycolsep{2pt}
\begin{eqnarray*}
 \left\{%
\begin{array}{lcrcl}
y^{\Delta}(t)=a(t)-b(t)e^{y(t)}-\displaystyle\frac{c(t)}{d(t)+m(t)e^{y(t)}},\,\,t\neq t_{k},\,\,t\in[t_{0}, +\infty)_{\mathbb{T}},\\
y(t_{k}^{+})=y(t_{k})\ln(1+\lambda_{k}),\,\,t=t_{k},\,\,k\in\mathbb{N}.
\end{array}\right.
\end{eqnarray*}}
We can easily see that $y(t)$ is a solution of system $\eqref{e12}$ and $ x_{\ast}-\epsilon\leq y(t)\leq x^{\ast}+\epsilon$ for $t\in\mathbb{T}$. Since $\epsilon$ is an arbitrary small positive number, it follows that $ x_{\ast}\leq y(t)\leq  x^{\ast}$ for $t\in\mathbb{T}$.
\end{proof}

\begin{theorem}\label{thm41}
Assume that $(H_{1})$-$(H_{4})$ hold. Suppose further that
\begin{itemize}
    \item  [$(H_{5})$]  $\gamma>0$ and $-\gamma\in\mathcal{R}^{+}$, where
    \begin{eqnarray*}
\gamma=2b^{l}e^{x_{\ast}}+{\displaystyle \frac{2c^{l}m^{l}e^{x_{\ast}}}{(d^{u}+m^{u}e^{x^{\ast}})^{2}}}-\bar{\mu}(b^{u})^{2}e^{2x^{\ast}}-{\displaystyle \frac{\bar{\mu}(c^{u})^{2}(m^{u})^{2}e^{2x^{\ast}}}{(d^{l}+m^{l}e^{x_{\ast}})^{4}}}-\frac{2\bar{\mu}b^{u}c^{u}m^{u}e^{2x^{\ast}}}{(d^{l}+m^{l}e^{x_{\ast}})^{2}}.
\end{eqnarray*}
\end{itemize}
Then $\eqref{e12}$ has a unique almost periodic solution $x(t)$, which is  uniformly
asymptotically stable.
\end{theorem}
\begin{proof}
From Lemma \ref{lem42}, there exists $x(t)$ such that $ x_{\ast}\leq x(t)\leq x^{\ast}$. Hence, $|x(t)|\leq K$, where $K=\max\{| x^{\ast}|, | x_{\ast}|\}$. Denote $\|x\|=\sup_{t\in\mathbb{T}}|x(t)|$.
 Suppose that
$x=x(t)$, $y=y(t)$ are any two positive solutions of system \eqref{e12}, then $\|x\|\leq K$, $\|y\|\leq K$. In view of system \eqref{e12}, we have
{\setlength\arraycolsep{2pt}
\begin{eqnarray}\label{e42}
 \left\{%
\begin{array}{lcrcl}
x^{\Delta}(t)=a(t)-b(t)e^{x(t)}-\displaystyle\frac{c(t)}{d(t)+m(t)e^{x(t)}},\,\,t\neq t_{k},\,\,t\in[t_{0}, +\infty)_{\mathbb{T}},\\
x(t_{k}^{+})=x(t_{k})\ln(1+\lambda_{k}),\,\,t=t_{k},\,\,k\in\mathbb{N},\\
y^{\Delta}(t)=a(t)-b(t)e^{y(t)}-\displaystyle\frac{c(t)}{d(t)+m(t)e^{y(t)}},\,\,t\neq t_{k},\,\,t\in[t_{0}, +\infty)_{\mathbb{T}},\\
y(t_{k}^{+})=y(t_{k})\ln(1+\lambda_{k}),\,\,t=t_{k},\,\,k\in\mathbb{N}.\\
\end{array}\right.
\end{eqnarray}}
Consider the Lyapunov function $V(t, x, y)$ on $\mathbb{T}^{+}\times\Omega\times\Omega$ defined by
\begin{eqnarray*}
V(t, x, y)=(x(t)-y(t))^{2}.
\end{eqnarray*}
It is easy to see that  there exist two constants $C_{1}>0$, $C_{2}>0$ such that  $(C_{1}\|x-y\|)^{2}\leq V(t, x, y)\leq (C_{2}\|x-y\|)^{2}$. Let $a, b\in C(\mathbb{R}^{+}, \mathbb{R}^{+})$, $a(x)=C_{1}^{2}x^{2}$, $b(x)=C_{2}^{2}x^{2}$, so the condition $(i)$ of Lemma \ref{lem41} is satisfied. Besides,
\begin{eqnarray*}
&&|V(t, x, y)-V(t, \bar{x}, \bar{y})|\\
&=&|(x(t)-y(t))^{2}-(\bar{x}(t)-\bar{y}(t))^{2}|\\
&\leq&|(x(t)-y(t))-(\bar{x}(t)-\bar{y}(t))||(x(t)-y(t))+(\bar{x}(t)-\bar{y}(t))|\\
&\leq&|(x(t)-y(t))-(\bar{x}(t)-\bar{y}(t))|\big(|x(t)|+|y(t)|+|\bar{x}(t)|+|\bar{y}(t)|\big)\\
&\leq&4K\big[|x(t)-\bar{x}(t)|+|y(t)-\bar{y}(t)|\big]\\
&=&L\big(\|x-\bar{x}\|+\|y-\bar{y}\|\big),
\end{eqnarray*}
where $L=4K$, so condition $(ii)$ of Lemma \ref{lem41} is also satisfied.

On the other hand for $t=t_{k}$, we have
\begin{eqnarray*}
V(t_{k}^{+}, x(t_{k}^{+}), y(t_{k}^{+}))&=&(x(t_{k}^{+})-y(t_{k}^{+}))^{2}\\
&=&[\ln(1+\lambda_{k})]^{2}(x(t_{k})-y(t_{k}))^{2}\\
&\leq&(x(t_{k})-y(t_{k}))^{2}\\
&=&V(t_{k}, x(t_{k}), y(t_{k})),
\end{eqnarray*}
then condition $(iii)$ of Lemma \ref{lem41} is also satisfied.

In view of system \eqref{e42}, we have
{\arraycolsep=2pt
\begin{eqnarray}\label{e43}
 \addtolength{\arraycolsep}{-3pt}
 \left\{%
\begin{array}{lcrcl}
(x(t)-y(t))^{\Delta}=-b(t)(e^{x(t)}-e^{y(t)})\\
\quad\quad\quad\quad\quad\quad\quad\,\,+{\displaystyle c(t)\bigg(\frac{1}{d(t)+m(t)e^{x(t)}}-\frac{1}{d(t)+m(t)e^{y(t)}}\bigg)},\,t\neq t_{k},\,t\in[t_{0}, +\infty)_{\mathbb{T}},\\
x(t_{k}^{+})-y(t_{k}^{+})=(x(t_{k})-y(t_{k}))\ln(1+\lambda_{k}),\,t=t_{k},\,k\in\mathbb{N}.
\end{array}\right.\quad
\end{eqnarray}}
For convenience, we denote $u(t)=x(t)-y(t)$. Using the mean value theorem we get
\begin{eqnarray}\label{e44}
e^{x(t)}-e^{y(t)}=e^{\xi(t)}(x(t)-y(t)),
\end{eqnarray}
\begin{eqnarray}\label{e45}
\frac{1}{d(t)+m(t)e^{x(t)}}-\frac{1}{d(t)+m(t)e^{y(t)}}=-\frac{m(t)e^{\zeta(t)}}{(d(t)+m(t)e^{\zeta(t)})^{2}}(x(t)-y(t)),
\end{eqnarray}
where $\xi(t)$ and $\zeta(t)$ lie between $x(t)$ and $y(t)$. Then, by use of \eqref{e44} and \eqref{e45}, \eqref{e43} can be written as
{\setlength\arraycolsep{2pt}
\begin{eqnarray}\label{e46}
 \left\{%
\begin{array}{lcrcl}
u^{\Delta}(t)=-b(t)e^{\xi(t)}u(t)-{\displaystyle \frac{c(t)m(t)e^{\zeta(t)}}{(d(t)+m(t)e^{\zeta(t)})^{2}}u(t)},\,\,t\neq t_{k},\,\,t\in[t_{0}, +\infty)_{\mathbb{T}},\\
u(t_{k}^{+})=u(t_{k})\ln(1+\lambda_{k}),\,\,t=t_{k},\,\,k\in\mathbb{N}.
\end{array}\right.
\end{eqnarray}}
Calculating the right derivative $D^{+}V^{\Delta}$ of $V$ along the solution of \eqref{e46} for $t\neq t_{k}$,
\begin{eqnarray*}
&&D^{+}V^{\Delta}(t, x, y)\\
&=&[2(x(t)-y(t))+\mu(t)(x(t)-y(t))^{\Delta}](x(t)-y(t))^{\Delta}\\
&=&[2u(t)+\mu(t)u^{\Delta}(t)]u^{\Delta}(t)\\
&=&
\bigg\{2u(t)+\mu(t)\bigg[-b(t)e^{\xi(t)}u(t)-{\displaystyle \frac{c(t)m(t)e^{\zeta(t)}}{(d(t)+m(t)e^{\zeta(t)})^{2}}u(t)}\bigg]\bigg\}\\
&&\times\bigg[-b(t)e^{\xi(t)}u(t)-{\displaystyle \frac{c(t)m(t)e^{\zeta(t)}}{(d(t)+m(t)e^{\zeta(t)})^{2}}u(t)}\bigg]\\
&=&
-2b(t)e^{\xi(t)}u^{2}(t)-{\displaystyle \frac{2c(t)m(t)e^{\zeta(t)}}{(d(t)+m(t)e^{\zeta(t)})^{2}}u^{2}(t)}+\mu(t)b^{2}(t)e^{2\xi(t)}u^{2}(t)\\
&&+{\displaystyle \frac{\mu(t)c^{2}(t)m^{2}(t)e^{2\zeta(t)}}{(d(t)+m(t)e^{\zeta(t)})^{4}}u^{2}(t)}+\frac{2\mu(t)b(t)m(t)c(t)e^{(\xi(t)+\zeta(t))}}{(d(t)+m(t)e^{\zeta(t)})^{2}}u(t)^{2}\\
&\leq&
-2b^{l}e^{x_{\ast}}u^{2}(t)-{\displaystyle \frac{2c^{l}m^{l}e^{x_{\ast}}}{(d^{u}+m^{u}e^{x^{\ast}})^{2}}u^{2}(t)}+\bar{\mu}(b^{u})^{2}e^{2x^{\ast}}u^{2}(t)+{\displaystyle \frac{\bar{\mu}(c^{u})^{2}(m^{u})^{2}e^{2x^{\ast}}}{(d^{l}+m^{l}e^{x_{\ast}})^{4}}u^{2}(t)}\\
&&+\frac{2\bar{\mu}b^{u}c^{u}m^{u}e^{2x^{\ast}}}{(d^{l}+m^{l}e^{x_{\ast}})^{2}}u(t)^{2}\\
&\leq&
-\bigg[2b^{l}e^{x_{\ast}}+{\displaystyle \frac{2c^{l}m^{l}e^{x_{\ast}}}{(d^{u}+m^{u}e^{x^{\ast}})^{2}}}-\bar{\mu}(b^{u})^{2}e^{2x^{\ast}}-{\displaystyle \frac{\bar{\mu}(c^{u})^{2}(m^{u})^{2}e^{2x^{\ast}}}{(d^{l}+m^{l}e^{x_{\ast}})^{4}}}-\frac{2\bar{\mu}b^{u}c^{u}m^{u}e^{2x^{\ast}}}{(d^{l}+m^{l}e^{x_{\ast}})^{2}}\bigg]u^{2}(t)\\
&\leq&
-\gamma u^{2}(t)\\
&=&-\gamma V(t, x, y).
\end{eqnarray*}
By  $(H_{5})$, we see that condition $(iv)$ of Lemma \ref{lem41} holds. Hence, according to Lemma \ref{lem41}, there exists a unique uniformly asymptotically stable almost periodic solution $x(t)$ of system $\eqref{e12}$, and $x(t)\in\Omega$.
\end{proof}

\section{An example}
\setcounter{equation}{0}
{\setlength\arraycolsep{2pt}}
 \indent

Consider the following  single-species system with impulsive effects on time scale $\mathbb{T}$:
{\setlength\arraycolsep{2pt}
\begin{eqnarray}\label{e51}
 \left\{%
\begin{array}{lcrcl}
x^{\Delta}(t)=a(t)-b(t)e^{x(t)}-\displaystyle\frac{c(t)}{d(t)+m(t)e^{x(t)}},\,\,t\neq t_{k},\,\,t\in[0, +\infty)_{\mathbb{T}},\\
x(t_{k}^{+})=x(t_{k})\ln(1+\lambda_{k}),\,\,t=t_{k},\,\,k\in\mathbb{N},
\end{array}\right.
\end{eqnarray}}
where $\mathbb{T}=\mathbb{R}$ or $\mathbb{T}=\mathbb{Z}$, and
\[a(t)=0.4-0.01\sin(\sqrt{2}t),\,\,b(t)=0.34,\,\,c(t)=0.009+0.001\cos(\sqrt{5}t),\]
\[d(t)=1.05+0.05\cos(\sqrt{5}t),\,\,m(t)=0.2+0.03\sin(\sqrt{3}t),\,\,\lambda_{k}=e^{(0.9)^{\frac{1}{2^{k}}}}-1,\,\,t_k=k.\]
 By calculating, we have
\[{a^u}=0.41,\,\,{a^l}=0.39,\,\, {b^u}={b^l}=0.34,\,\,{c^u}=0.01,\,\,{c^l}=0.008,\]
\[{d^u}=1.1,\,\,{d^l}=1,\,\, {m^u}=0.23,\,\, {m^l}=0.17,\,\,r\approx0.949,\]
so we obtain
\[
x^{\ast}=\frac{ a^{u}-b^{l}}{b^{l}}\approx0.206,\,\,
x_{\ast}=\ln\bigg(\frac{(a^{l}-c^{u})\times0.949}{b^{u}}\bigg)\approx0.059.\]
When  $\mathbb{T}=\mathbb{R}$, then $\mu(t)=0$ and $\gamma\approx0.723$. When   $\mathbb{T}=\mathbb{Z}$, then $\mu(t)=1$ and $\gamma\approx0.547$. So, in both cases $\mathbb{T}=\mathbb{R}$ and $\mathbb{T}=\mathbb{Z}$, all conditions in Theorem \ref{thm41} are satisfied, thus \eqref{e51} has a unique positive almost periodic solution, which is uniformly asymptotically stable.

\section{Conclusion}
\indent

In this paper, some new comparison theorems and a Massera type theorem for impulsive dynamic equations on time scales are established. Based on these results, the existence and uniformly asymptotic stability of unique positive almost periodic solution of a single-species system with impulsive effects on time scales is obtained. Our results of this paper are completely new and can be used to study other types impulsive dynamic equation models on time scales.

\end{document}